\documentclass[reqno0,fleqn,12pt]{amsart}
\pdfoutput=1
\usepackage{bbm}
\usepackage[nodate]{datetime}
\usepackage[hmargin=25mm,top=25mm,bottom=25mm,a4paper]{geometry}
\usepackage{color}
\usepackage{subfig}
\usepackage{fancyhdr}
\usepackage{amsmath,amssymb,amsfonts,amsthm}
\usepackage{amstext}
\usepackage{amsmath}
\usepackage{float}
\usepackage{amssymb}
\usepackage{enumerate}
\usepackage{amsbsy}
\usepackage{amsopn}
\usepackage{bbm,amsthm}
\usepackage{amscd}
\usepackage{amsxtra}
\usepackage{mathtools}
\usepackage{hyperref}%include [hidelinks] after brackets to hide links
\usepackage{cleveref}
\usepackage{todonotes}

\newtheorem{theorem}{Theorem}[section]
\newtheorem*{theorem*}{Theorem}
\newtheorem{lemma}{Lemma}[section]
\newtheorem{proposition}{Proposition}[section]

\newtheorem{conjecture}{Conjecture}[section]

\newtheorem{remark}{Remark}

\newtheorem{question}{Question}

\usepackage{tikz-3dplot}
\usepackage{tikzscale}
\usepackage{multicol}

\begin{document}
\title{Modified Dirichlet character sums over the \texorpdfstring{$k$}{k}-free integers}
\author{Caio Bueno}
\begin{abstract}
    The main question of this paper is the following: how much cancellation can the partial sums restricted to the $k$-free integers up to $x$ of a $\pm 1$ multiplicative function $f$ be in terms of $x$? Building upon the recent paper by Q. Liu, Acta Math. Sin. (Engl. Ser.) 39 (2023), no. 12, 2316–2328, we prove that under the Riemann Hypothesis for quadratic Dirichlet $L$-functions, we can get $x^{1/(k+1)}$ cancellation when $f$ is a modified quadratic Dirichlet character, i.e., $f$ is completely multiplicative and for some quadratic Dirichlet character $\chi$, $f(p)=\chi(p)$ for all but a finite subset of prime numbers. This improves the conditional results by Aymone, Medeiros and the author cf. Ramanujan J. 59 (2022), no. 3, 713–728.
\end{abstract}
\maketitle
\section{Introduction}\label{Introduction}

\subsection{Main result and recent progress}
We say that $n$ is a $k$-free number ($k\geq 2$) if it is a positive integer such that no prime power $p^k$ divides $n$. Let $\mu^{(k)}$ denote the indicator function of the $k$-free numbers, i.e.
    \begin{align*}
        \mu^{(k)}(n)=
            \begin{cases}
                1       &\text{ if $n$ is a $k$-free number,}\\
                0       &\text{ otherwise.}
            \end{cases} 
    \end{align*}

An arithmetic function $f:\mathbb{N}\to \mathbb{C}$ is said to be multiplicative if $f(mn)=f(m)f(n)$, for any $m,n\in\mathbb{N}$ such that $(m,n)=1$. Moreover, we say that $f$ has support on the $k$-free integers if $f(n)=0$, whenever $n$ is not a $k$-free number.

In \cite{Aymone-Bueno-Medeiros-kfree}, Aymone, Medeiros and the author asked the following question:
    \begin{question}
        For which exponents $\alpha>0$ does there exist a multiplicative function $f$ that satisfies the following properties: $f$ has support on the $k$-free integers, takes values $\pm 1$ at prime numbers, and its partial sums up to $x$ are $\ll x^\alpha$?
    \end{question}

In the present paper we prove the following:
    \begin{theorem}\label{main-theorem}
        Assume the Generalized Riemann hypothesis. There exists a multiplicative function $f:\mathbb{N}\to\{-1,0,+1\}$ with support on the $k$-free integers such that, at prime numbers, $f$ takes values in $\{-1,+1\}$ and we have
            \begin{align*}
                \sum_{n\leq x}f(n)\ll x^{1/(k+1)+\varepsilon},\forall \varepsilon>0.
            \end{align*}
    \end{theorem}

The class of functions we will consider are the so called modified Dirichlet characters. We say that a function $g:\mathbb{N}\to \{z\in\mathbb{C} : |z|=1\}$ is a modified Dirichlet character if $g$ is completely multiplicative, i.e., $g(mn)=g(m)g(n),\forall m,n\in\mathbb{N}$, and agrees with some Dirichlet character $\chi$ for all but a finite number of primes.

It was first established by Aymone in \cite{Aymone-2/5} that, if we let $\chi$ be a real, non-principal Dirichlet character of modulus $q$ and $g_\chi$ a modified character related to $\chi$, by defining $f=\mu^2g_\chi$ and assuming the Riemann Hypothesis, we have
    \begin{align*}
        \sum_{n\leq x}f(n)\ll_\varepsilon x^{2/5+\varepsilon},\forall\varepsilon>0,
    \end{align*}
where the subindex in $\ll_\varepsilon$ means that the implicit constant might depend on the parameter $\varepsilon$.

Unconditionally, Aymone proved in the same paper that
    \begin{align*}
        \sum_{n\leq x}f(n)\ll_\delta x^{1/2}\exp(-\delta(\log x)^{1/4}),
    \end{align*}
for some constant $\delta>0$.

Later, adapting the methods used in \cite{Aymone-2/5}, the result was generalized to the $k$-free integers by Aymone, Medeiros and the author. It was shown in \cite{Aymone-Bueno-Medeiros-kfree}, for $f=\mu^{(k)}g_\chi$ and assuming the Riemann Hypothesis for $L$-functions, that
    \begin{align*}
        \sum_{n\leq x}f(n)\ll_{\varepsilon}x^{1/(k+\frac{1}{2})+\varepsilon},\forall\varepsilon>0.
    \end{align*}

Unconditionally, it was obtained
    \begin{align*}
        \sum_{n\leq x}f(n)\ll_\delta x^{1/k}\exp(-\delta (\log x)^{1/4}),
    \end{align*}
for some constant $\delta>0$. We remark that the previous generalization coincides with the result obtained by Aymone in the case when $k=2$ and also that, in the proof of both the unconditional results above, the classical zero-free region for the Riemann zeta function and also for $L$-functions in the $k$-free case, which are essentially the same, was used.

More recently Liu \cite{Liu-1/3} showed, for the squarefree case, that
    \begin{align*}
        \sum_{n\leq x}f(n)\ll_{q,\varepsilon} x^{1/3+\varepsilon},
    \end{align*}
for any $\varepsilon>0$, under the Riemann Hypothesis. For this conditional result, the strategy employed by the author in order to estimate the tail of a Dirichlet series is elegantly adapted from the idea of Montgomery and Vaughan \cite{Distribution-of-2free-montgomeryvaughan} and the tools used are solid. In this paper, we build upon Liu's work to generalize the above result to the $k$-free integers, presenting the proofs in a slightly different manner while keeping the same approach to the problem.

Liu also proved that, unconditionally, there is a constant $\delta>0$ such that
    \begin{align*}
        \sum_{n\leq x}f(n)\ll_{q} x^{1/2}\exp\bigg(\frac{-\delta (\log x)^{3/5}}{(\log\log x)^{1/5}}\bigg).
    \end{align*}
The term inside the exponential was improved by exploring the Vinogradov-Korobov type of zero-free regions for the Riemann zeta function, but otherwise there is no power-saving on the main term compared to the previous results.

    \begin{remark}\label{Conjecture}
        In \cite{Aymone-Bueno-Medeiros-kfree} the authors conjectured that the real magnitude of the partial sums of $f=\mu^{(k)}g_\chi$, up to $x$, is $\ll x^{1/(2k)+\varepsilon}$ and this conjecture remains open.
    \end{remark}

\subsection{Motivation}

Before proceeding, we shall first present the motivation behind this topic.

The study of this type of problem got more traction in recent years after Tao's solution of the Erd\H{o}s Discrepancy Problem (EDP), in 2016 - see \cite{EDP-Tao}. In short, the EDP questioned if the following
    \begin{align}\label{discrepancy-definition}
        \sup_{n,d\in\mathbb{N}}\bigg|\sum_{k=1}^nh(kd)\bigg|
    \end{align}
is infinite or not, for any function $h:\mathbb{N}\to\{-1,1\}$. The supremum \eqref{discrepancy-definition} above is called the discrepancy of the function $h$ and Tao showed in \cite{EDP-Tao} that it is, in fact, infinite.

Partial sums of non-principal Dirichlet characters are particularly interesting in view of the problem presented. In this context, we say that these functions are ``extreme examples" or ``near-counterexamples" to the EDP and this is due to the fact that they constitute a non-trivial function that nearly satisfies the hypothesis of the EDP but still have bounded partial sums, and thus, bounded discrepancy. This can be seen when we recall that $\chi$ is completely multiplicative and apply, for example, the Pólya-Vinogradov inequality. Therefore,
    \begin{align*}
        \bigg|\sum_{k=1}^n \chi(kd)\bigg|=\bigg|\sum_{k=1}^n\chi(d)\chi(k)\bigg|\leq \bigg|\sum_{k=1}^n\chi(k)\bigg|\leq 6\sqrt{q}\log q,
    \end{align*}
for any $d\in\mathbb{N}$ and any fixed modulus $q$ of $\chi$. It follows that the discrepancy of $\chi$ is at most $6\sqrt{q}\log q$. Needless to say, this is not a real counter-example, since $\chi(p)=0$ if $p$ divides the modulus $q$. However, the preceding illustrate how impactful the hypothesis that $|h|=1$ in this context is and, by failling to verify this hypothesis only for a finite subset of primes, we already can have bounded discrepancy.

Historically, understanding how the partial sums of certain multiplicative functions behaves is notably significant. The EDP above is an example and another indication of this importance is the well-known result that the Riemann Hypothesis is equivalent to the assertion that the partial sums of the M\"obius function up to $x$ oscillates at most $x^{1/2+\varepsilon},\forall \varepsilon>0$, as $x$ goes to infinity.

Seemingly motivated by the aforementioned discussion and the fact that the solution of the EDP reduces to completely multiplicative functions that pretends\footnote{Given two multiplicative functions $f$ and $g$ such that $|f|,|g|\leq 1$, we say that $f$ pretends to be $g$ if $\mathbb{D}(f,g,x)^2=\sum_{p\leq x}\frac{1-\Re(f(p)\overline{g(p)})}{p}\ll 1$.} to be a real Dirichlet character, Aymone \cite{Aymone-2/5} proposed the investigation of examples of multiplicative functions supported on the squarefree integers with small partial sums and naturally looked first for functions that are close to Dirichlet characters, the modified characters.

\subsection{Other related results}

Continuing to follow this line of problems, we now briefly comment on other related results. For instance, one could investigate the partial sums of a Dirichlet character over the $k$-free integers or even the partial sums of a modified Dirichlet character itself.

The former was studied by Liu and Zhang \cite{Liu-Zhang-sqrfree-sqrfull} and later by Munsch \cite{Munsch}. Precisely, consider a (general) non-principal Dirichlet character $\chi$ (mod $q$) over the squarefree integers, i.e., $\mu^2\chi$. In \cite{Liu-Zhang-sqrfree-sqrfull}, the authors proved the following mixed bound:
    \begin{align}\label{Liu-Zhang}
        \sum_{n\leq x}\mu^2(n)\chi(n)\ll x^{1/2+\varepsilon}q^{9/44+\varepsilon},
    \end{align}
for any $\varepsilon>0$.

With the aid of two classical results on partial sums of Dirichlet characters, one being the Pólya-Vinogradov inequality, which we already commented above and states that
    \begin{align}\label{Polya-Vinogradov}
        \bigg|\sum_{n\leq x}\chi(n)\bigg|\leq 6\sqrt{q}\log q
    \end{align}
and the other is Burgess inequality (see \cite{Burgess}), which is
    \begin{align}\label{Burgess}
        \bigg|\sum_{n\leq x}\chi(n)\bigg|\ll x^{1/2}q^{3/16+\varepsilon},
    \end{align}
the bound \eqref{Liu-Zhang} was improved by Munsch \cite{Munsch}, where he obtained
    \begin{align*}
        \sum_{n\leq x}\mu^2(n)\chi(n)\ll
            \begin{cases}
                x^{1/2}q^{1/4}(\log q)^{1/2}        &\text{(using \eqref{Polya-Vinogradov})},\\
                x^{1/2}(\log x)q^{3/16+\varepsilon},\forall\varepsilon>0        &\text{(using \eqref{Burgess})},\\
                x^{1/2}(\log x)q^{3/16}(\log q)^{1/2}       &\text{(for $q$ prime, also using \eqref{Burgess})}.
            \end{cases}
    \end{align*}

For the latter case, that is, the study of partial sums of modified characters $g$, Aymone \cite{Aymone-omega-bounds} was interested on omega\footnote{We say that $f(x)=\Omega(g(x))$ if $\limsup_{x\to\infty}|\frac{f(x)}{g(x)}|>0$.} results for this type of sums and it was shown, under some additional hypothesis, that
    \begin{align*}
        \sum_{n\leq x}g(n)=\Omega((\log x)^{|S|}).
    \end{align*}

Since this is a very unexplored topic up to this date, an interested reader can see \cite{Aymone-omega-bounds} and the discussions therein. For omega results on modified Dirichlet characters sums over squarefree integers, see \cite{Aymone-2/5,Klurman-Mangerel-Pohoata-Teravainen,Liu-1/3}.

Going forward, another question that is closely related to the ones already discussed here is the following: is there a multiplicative function $h:\mathbb{N}\to \{-1,1\}$ such that the sum over the squarefree integers $\sum_{n\leq x}\mu^2(n)h(n)$ is uniformly bounded by some constant? Aymone conjectured in \cite{Aymone-edp-cubefree} that, for any such multiplicative function $h$, we have
    \begin{align*}
        \sum_{n\leq x}\mu^2(n)h(n)=\Omega(x^{1/4-\varepsilon}),
    \end{align*}
for any $\varepsilon>0$. More generally, it is believed that
    \begin{conjecture}[Aymone \cite{Aymone-edp-cubefree}]
        If $h:\mathbb{N}\to\{-1,1\}$ is multiplicative, then
            \begin{align*}
                \sum_{n\leq x}\mu^{(k)}(n)h(n)=\Omega(x^{1/(2k)-\varepsilon}),\forall\varepsilon>0.
            \end{align*}
    \end{conjecture}

Towards this conjecture, Aymone \cite{Aymone-edp-cubefree} and Klurman, Mangerel, Pohoata and Teräväinen \cite{Klurman-Mangerel-Pohoata-Teravainen} proved, independently and with different methods, that the partial sums over the squarefree integers of any multiplicative function that takes values in $\{-1,1\}$ must be unbounded. Moreover, Aymone also showed that if we now consider a completely multiplicative function taking values in $\{-1,1\}$, its partial sums over the cubefree integers, are unbounded as well.

Finally, as pointed out in \cite{Aymone-edp-cubefree}, if $h:\mathbb{N}\to\{-1,1\}$ is completely multiplicative, one can adapt the proofs contained there to show that the partial sums of $h$ over the $k$-free integers are unbounded. However, when $h$ is only multiplicative, this problem seems much harder.

\section{Notation}

Here we fix a few notations that will be recurrent throughout the paper.

The function $\omega(n)$ will detone the number of distinct prime factors of the integer $n$ and $\pi(x)$ is the classical prime counting function, which gives the number of prime numbers up to a $x$. The letter $p$ will always denote a prime number.

We use $f(x)\ll g(x)$ to say that there exists a constant $C$ and a real number $x_0$ such that $|f(x)|\leq C g(x),\forall x\geq x_0$. Sometimes we write a subindex, e.g. $\ll_\varepsilon$, to indicate that the implicit constant depends on the parameter $\varepsilon$. Moreover, we use $O$ and $\ll$ (resp. $O_\varepsilon$ and $\ll_\varepsilon$) interchangeably. The notation $f(x)=o(g(x))$ means that, for every $\varepsilon>0$, there a constant $x_0$ such that $|f(x)|\leq \varepsilon g(x),\forall x\geq x_0$.

We say that an integer $n$ is $k$-free if it is not divisible by any prime power $p^k$, where $k\in\mathbb{Z}_{\geq 2}$. By saying that a function $f$ has support on the $k$-free integers, we mean that $f(n)=0$ if $n$ is not a $k$-free integer.

The fractional and integer part of a real number $x$ is denoted by $\{x\}$ and $[x]$, respectively. Unless stated otherwise, $s$ is a complex number with real part $\Re(s)\coloneqq\sigma$ and imaginary part $\Im(s)\coloneqq t$.

The cardinality of a set $A$ is denoted by $|A|$.

\section{Main result and discussion about the problem\label{main-result-section}}

This section is dedicated to discuss some subtleties about our proof. We repeat some definitions that were already presented in \Cref{Introduction}.

Since we intend to work with real Dirichlet characters $\chi$, we recall that if $p$ is a prime number that divides the modulus $q$, then, by definition, we have $\chi(p)=0$. Therefore, we define the modified Dirichlet character $g_\chi:\mathbb{N}\to\{-1,1\}$ as the completely multiplicative function with the following rule:
    \begin{align}\label{modified-char-def}
        g_\chi(p)=
            \begin{cases}
                \chi(p)     &\text{ if $p\nmid q$,}\\
                1           &\text{ if $p\mid q$}.\\
            \end{cases}
    \end{align}

We note that the choice of the sign $+1$ when $p\mid q$ is arbitrary, and we are merely following the convention in previous works on the subject. Choosing $-1$ instead would not affect the results presented in this paper and would require only minor modifications in our proofs.
    
Let $k\geq 2$ be an integer and $f=\mu^{(k)}g_\chi$, where $\mu^{(k)}$ is the indicator function of the $k$-free integers. By this definition, we note that $f:\mathbbm{N}\to\{-1,0,1\}$ is a multiplicative function supported on the $k$-free integers.

We recall that the Riemann Hypothesis (RH) for the Riemann zeta function and the Generalized Riemann Hypothesis (GRH) for Dirichlet $L$-functions are, respectively, the conjecture that all the non-trivial zeros of $\zeta(s)$ and of $L(s,\chi)$ have real part equal to $\frac{1}{2}$.

    \begin{remark}
        Note that for \Cref{main-theorem}, we are assuming the RH and the GRH. In fact, we will observe that if $k$ is even, we need to assume only the Riemann Hypothesis and if $k$ is odd, we need only the Generalized Riemann Hypothesis.
    \end{remark}

To prove \Cref{main-theorem}, we follow the strategy of \cite{Liu-1/3} for functions supported on the squarefree integers (in our notation, when $k=2$). So first, we express the generating series of $f$ as more familiar series and we do that by means of the Euler product formula.

Let $F(s)\coloneqq\sum_{n=1}^\infty\frac{f(n)}{n^s}$ be the Dirichlet series associated to the function $f$ defined previously. Since $|f(n)|=|\mu^{(k)}(n)g_\chi(n)|\leq 1, \forall n$, $F(s)$ is absolutely convergent for every $s\in\mathbbm{C}$ such that $\Re(s)>1$. Since $f$ is multiplicative, applying the Euler product formula, we have for $\Re(s)>1$ that
    \begin{align*}
        F(s)&=\prod_p\sum_{k=0}^\infty\frac{f(p^k)}{p^{ks}}=\prod_p\bigg(1+\frac{\mu^{(k)}(p)g_\chi(p)}{p^s}+\frac{\mu^{(k)}(p^2)g_\chi(p^2)}{p^{2s}}+\cdots\bigg)\\
        &=\prod_p\bigg(1+\frac{g_\chi(p)}{p^s}+\frac{g_\chi(p^2)}{p^{2s}}+\cdots+\frac{g_\chi(p^{k-1})}{p^{(k-1)s}}\bigg)\\
        &=\prod_{p\mid q}\bigg(1+\frac{1}{p^s}+\cdots+\frac{1}{p^{(k-1)s}}\bigg)\prod_{p\nmid q}\bigg(1+\frac{\chi(p)}{p^s}+\cdots+\frac{\chi(p^{k-1})}{p^{(k-1)s}}\bigg).
    \end{align*}

Since $(1+\frac{1}{p^s}+\cdots+\frac{1}{p^{(k-1)s}})(1-\frac{1}{p^s})=(1-\frac{1}{p^{ks}})$ and $(1+\frac{\chi(p)}{p^s}+\cdots+\frac{\chi(p^{k-1})}{p^{(k-1)s}})(1-\frac{\chi(p)}{p^s})=(1-\frac{\chi(p)^k}{p^{ks}})$, it follows that
    \begin{align*}
        F(s)&=\frac{\prod_{p\nmid q}(1-\frac{\chi(p)^k}{p^{ks}})}{\prod_{p\nmid q}(1-\frac{\chi(p)}{p^s})}\frac{\prod_{p\mid q}(1-\frac{1}{p^{ks}})}{\prod_{p\mid q}(1-\frac{1}{p^s})}\\
        &=L(s,\chi)P(s)\prod_{p\nmid q}\bigg(1-\frac{\chi(p)^k}{p^{ks}}\bigg)\prod_{p\mid q}\bigg(1-\frac{1}{p^{ks}}\bigg),
    \end{align*}
where $P(s)\coloneqq\prod_{p\mid q}(1-\frac{1}{p^s})^{-1}$, which depends on the modulus $q$.

Now, if $k$ is even and $p\nmid q$, $\chi(p)^k=1$. Then,
    \begin{align*}
        F(s)=L(s,\chi)\frac{P(s)}{\zeta(ks)}.
    \end{align*}

If $k$ is odd, $\chi(p)^k=\chi(p)$. Then,
    \begin{align*}
        F(s)=\frac{L(s,\chi)}{L(ks,\chi)}\frac{P(s)}{P(ks)}.
    \end{align*}

We treat the two cases separately.

A comment on the product $P(s)$ above is due. Note that the function $P(s)$ has simple poles at the imaginary numbers $s=i\frac{2\pi t}{\log p}, t\in\mathbb{Z}\backslash \{0\}, p \mid q$, which are the zeros of $1-\frac{1}{p^s}$ and a pole of order $\omega(q)$ at $s=0$.

We also note that the function $P(s)$ extends analytically to $\Re(s)>0$. Therefore, if the Riemann Hypothesis (resp. Generalized Riemann Hypothesis) is true, the function $\frac{L(s,\chi)P(s)}{\zeta(ks)}$ (resp. $\frac{L(s,\chi)P(s)}{L(ks,\chi)P(ks)}$) has analytic continuation to the half-plane $\{s\in\mathbb{C} : \Re(s)>\frac{1}{2k}\}$. This fact by itself doesn't imply that the conjecture we mentioned in \Cref{Conjecture} is true, but we could still have room to improve the result of \Cref{main-theorem} in a meaningful way.

    \begin{remark}\label{remarkLH}
        To prove \Cref{main-theorem}, the RH and the GRH is crucial in order to stay away from the poles of $1/\zeta(ks)$ and $1/L(ks,\chi)$. One consequence of the RH is the Lindel\"of Hypothesis (LH) and a primary consequence of the latter conjecture is that $\zeta(\frac{1}{2}+it)\ll |t|^\varepsilon,\forall\varepsilon>0$ (in particular for every $\sigma\geq \frac{1}{2}$). As a consequence of the proof that the RH implies the LH, we also obtain the bound $1/\zeta(s)\ll_{\sigma_0,\varepsilon} |t|^\varepsilon,\forall\varepsilon>0$ and every $\Re(s)\geq \sigma_0> 1/2$. The same can be said about the GRH: we have that the GRH implies the LH for $L$-functions, and thus $L(s,\chi),1/L(s,\chi)\ll_{\sigma_0,\varepsilon} |qt|^\varepsilon,\forall\varepsilon>0$ and $\Re(s)\geq \sigma_0>1/2$.
    \end{remark}

\section{Auxiliary lemmas and preliminary results}

An essential tool for our proofs is the famous Perron Formula. The version of Perron's Theorem we will apply is stated in the following lemma:
    \begin{lemma}[\cite{Montgomery-Vaughan}, Theorem 5.1 and Corollary 5.3, pp. 138-140]\label{perron-formula-montgomery-vaughan}
        Let $s=\sigma+it$ and $A(s)=\sum_{n=1}^\infty\frac{a_n}{n^s}$. If $\sigma_0>\max\{0,\sigma_a\}$ and $x>0$, then
            \begin{align*}
                \sum_{n\leq x}a_n=\frac{1}{2\pi i}\int_{\sigma_0-iT}^{\sigma_0+iT}A(s)\frac{x^s}{s}\,ds+R(x),
            \end{align*}
        where
            \begin{align*}
                R(x)\ll \sum_{\substack{x/2<n<2x \\ n\neq x}}|a_n|\min\bigg\{1,\frac{x}{T|x-n|}\bigg\}+\frac{x^{\sigma_0}+4^{\sigma_0}}{T}\sum_{n=1}^\infty\frac{|a_n|}{n^{\sigma_0}}.
            \end{align*}
    \end{lemma}

Now we state a lemma regarding the second moment of $L(s,\chi)$:
    \begin{lemma}\label{second-moment-l-function}
        Let $\chi$ be a real non-principal Dirichlet character modulo $q$ and $\sigma\geq \frac{1}{2}$. We have,
            \begin{align*}
                \int_{-T}^{T}|L(\sigma+it,\chi)|^2\,dt\ll_q T\log T,
            \end{align*}
        as $T\to\infty$.
    \end{lemma}
Results of this type (with explicit dependencies on $q$) are proved in \cite{Ramachandra} and later with more precise error terms in \cite{Rane}. For our purposes, we use only the main term, which is sharp and should be enough.

The next result is essentially Lemma 4.3 of \cite{Liu-1/3}.
    \begin{lemma}\label{lema-integral-L}
        Let $\chi$ be a real non-principal Dirichlet character modulo $q$. Uniformly in the strip $1/2\leq\sigma\leq 2$, we have
            \[
                \int_{-T}^T\frac{|L(\sigma+it,\chi)|}{|\sigma+it|} \,dt\ll (\log T)^{3/2},
            \]
        as $T\to \infty$.
    \end{lemma}
    \begin{proof}
        By Cauchy-Schwarz inequality, we have
            \begin{align*}
                \bigg(\int_{-T}^{T}\frac{|L(\sigma+it,\chi)|}{|\sigma+it|}\,dt\bigg)^2\leq\bigg(\int_{-T}^{T}\frac{1}{|\sigma+it|}\,dt\bigg)\bigg(\int_{-T}^T\frac{|L(\sigma+it,\chi)|^2}{|\sigma+it|}\,dt\bigg).
            \end{align*}

        The first integral on the RHS is $\ll\int_{-T}^T\frac{1}{|\sigma+it|}\,dt\ll \int_1^T\frac{1}{|\sigma+it|}\,dt\ll \log T$.

        For the second integral, using \Cref{second-moment-l-function} and integration by parts, we have
            \begin{align*}
                \int_{-T}^T\frac{|L(\sigma+it,\chi)|^2}{|\sigma+it|}\, dt\ll \frac{T\log T}{T}+(\log T)^2\ll (\log T)^2.
            \end{align*}

        It follows that
            \begin{align*}
                \int_{-T}^T\frac{|L(\sigma+it,\chi)|}{|\sigma+it|}\, dt\ll (\log T)^{1/2}\log T\ll (\log T)^{3/2},
            \end{align*}
        which proves the lemma.
    \end{proof}

An additional lemma that will be needed later in the proof of the main theorem is the following:
    \begin{lemma}\label{lemma-abs-value-inside-sum}
        Let $h(n)$ be the coefficients of the Dirichlet series $\frac{P(s)}{\zeta(ks)}$. We have,
            \begin{align*}
                \sum_{n\leq x}|h(n)|\ll_q x^{1/k}(\log x)^{\pi(q)}.
            \end{align*}
    \end{lemma}
    \begin{proof}
        Let $\mathcal{N}=\{n : p|n \implies p|q\}$, where $q$ is the modulus of the real non-principal Dirichlet character linked to the function $h(n)$. Defining $\sum_{n=1}^\infty \frac{h(n)}{n^s}=\frac{P(s)}{\zeta(ks)}$, it follows that the coefficients $h(n)$ can be written as the Dirichlet convolution $h(n)=(\nu\ast\mathbbm{1}_{\mathcal{N}})(n)$, where $\nu(n)=
        \begin{cases}
            \mu(d)  &\text{  if $n=d^k$,}\\
            0       &\text{  otherwise.}\\
        \end{cases}$
    
        Then, we have
            \begin{align*}
                |h(n)|=|(\nu\ast\mathbbm{1}_{\mathcal{N}})(n)|\leq \sum_{\substack{a^kb=n \\ b\in \mathcal{N}}}|\mu(a)|\leq\sum_{\substack{a^kb=n \\ b\in \mathcal{N}}}1.
            \end{align*}

        Hence,
            \begin{align}\label{equation-modulus-h}
                \sum_{n\leq x}|h(n)|\leq \sum_{n\leq x}\sum_{\substack{a^kb=n \\ b\in \mathcal{N}}}1=\sum_{a\leq x^{1/k}}\sum_{\substack{b\leq \frac{x}{a^k} \\ b\in\mathcal{N}}}1.
            \end{align}

        In order to estimate the LHS, we need to bound the inner sum on the RHS, that is, sums of the type $\sum_{\substack{n\leq x \\ n \in\mathcal{N}}}1$. Note that if $q=p_1^{\alpha_1}p_2^{\alpha_2}\cdots p_K^{\alpha_K}$ and $n\in\mathcal{N}$, we have that $n$ can be decomposed as $n=p_1^{\beta_1}p_2^{\beta_2}\dots p_r^{\beta_r}$, where $r\leq K$ (although we have no restriction on the size of $\beta_i$ in comparison to $\alpha_i$). Let $S(x,y)$ be the set of integers less than or equal to $x$, with all prime factors $\leq y$. It follows that if $n\in\mathcal{N}\cap [1,x]$, then $n\in S(x,q)$.

        Therefore, the following bound holds $|\mathcal{N}\cap [1,x]|\leq |S(x,q)|$, that is, our sum can be bounded by the cardinality of the set $S(x,q)$. Elements in $S(x,y)$ are called $y$-smooth integers and the number of such elements is denoted in the literature by $\Psi(x,y)$. We also have, by a result of Granville (see \cite{Smooth-numbers}), that
            \begin{align*}
                \Psi(x,y)\ll_y (\log x)^{\pi(y)},
            \end{align*}
        where $\pi(y)$ counts the number of prime numbers up to $y$.
        
        Thus, applying this result to \eqref{equation-modulus-h},
            \begin{align*}
                \sum_{n\leq x}|h(n)|\ll_q\sum_{a\leq x^{1/k}} (\log x)^{\pi(q)}\ll_q x^{1/k}(\log x)^{\pi(q)},
            \end{align*}
        as $x\to\infty$.
    \end{proof}

\section{The main propositions}

As discussed at the end of \Cref{main-result-section}, our strategy is to work separately when $k$ is even or odd. We start with the even case.

\subsection{Proposition when \texorpdfstring{$k$}{k} is even}

Let $\sum_{n=1}^\infty\frac{h(n)}{n^s}$ be the Dirichlet series of the function $\frac{P(s)}{\zeta(ks)}$. For $s=\sigma+it\in\mathbbm{C}$ such that $\sigma\geq \frac{1}{k}+\delta, \delta>0$, define
    \begin{align*}
        H_y(s)=\frac{P(s)}{\zeta(ks)}-\sum_{n\leq y}\frac{h(n)}{n^s},
    \end{align*}
for a fixed $y\geq 1$. Our goal is to give an estimate for the partial sums of this tail, as $y$ goes to infinity, and that naturally is $\sigma$-dependent.

   \begin{proposition}\label{main-prop-keven}
        Let $H_y$ be defined as above and assume the Riemann Hypothesis. We have in the half-plane $\Re(s)\geq \frac{1}{k}+\varepsilon$,
            \begin{align*}
                H_y(s)\ll_{q,\varepsilon} y^{1/(2k)-\sigma+\varepsilon}|t|^\varepsilon,
            \end{align*}
        for any $\varepsilon>0$ and sufficiently large $t$, as $y\to\infty$.
    \end{proposition}
    \begin{proof}
        Our aim is to apply Perron's formula to the partial sums up to $y$, of $\frac{P(s+w)}{\zeta(k(s+w))}$, for a fixed $s$. First we note that this Dirichlet series converges absolutely in the half-plane $\{z\in\mathbb{C} : \Re(z)>\frac{1}{k}\}$, therefore, we can use \Cref{perron-formula-montgomery-vaughan} to its partial sums, as long as we integrate over the line segment with real part $\sigma_0>\max\{0,\frac{1}{k}-\sigma\}$. We could go arbitrarily close to the imaginary axis, e.g. $\sigma_0=\delta$, for any $\delta>0$, but that is not necessary.
        
        Let $\sigma\geq\frac{1}{k}+\varepsilon$ and take $\sigma_0=1$, say. If $w=u+iv$, we have
            \begin{align}\label{first-perron-application-keven}
                \sum_{n\leq y}\frac{h(n)}{n^s}=\frac{1}{2\pi i}\int_{1-iV}^{1+iV}\frac{P(s+w)}{\zeta(k(s+w))}\frac{y^w}{w}\,dw+R(y),
            \end{align}
        where
            \begin{align*}
                R(y)\ll \sum_{\substack{\frac{y}{2}<n<2y \\ n\neq y}}\bigg|\frac{h(n)}{n^s}\bigg|\min\bigg\{1,\frac{y}{V|y-n|}\bigg\}+\frac{4+y}{V}\sum_{n=1}^\infty\frac{|h(n)|}{n^{1+\sigma}}.
            \end{align*}
            
        Since $\sigma\geq \frac{1}{k}+\varepsilon$, the second term above is $\ll\frac{y}{V}$. For the first term, when $|n-y|<1$, we choose the first member of the minimum and for every other value of $n$ we choose the second member. Thus, since $\frac{|h(n)|}{n^\sigma}\ll 1$, we have
            \begin{align*}
                R(y)\ll 2+\frac{y}{V}\sum_{1\leq m\leq y}\frac{1}{m}+\frac{y}{V}\ll \frac{y\log y}{V}.
            \end{align*}
        
        Let $\alpha=\frac{1}{2k}-\sigma+\varepsilon$. We will shift the original contour of integration from the vertical line segment $1\pm iV$ to the piecewise linear path that goes from $1+iV$ to $\alpha+iV$, the vertical line segment from $\alpha+iV$ to $\alpha-iV$ and the horizontal line segment from $\alpha -iV$ to $1-iV$.

        Note that, since $\sigma\geq\frac{1}{k}+\varepsilon$, we have over the path of integration that $\Re(k(s+w))\geq k(\sigma+\alpha)=\frac{1}{2}+k\varepsilon$. Since we are assuming the Riemann Hypothesis, the poles of $\frac{1}{\zeta(k(s+w))}$ at the critical strip are located on the line $\{z\in\mathbb{C} : \Re(z)=\frac{1}{2}\}$, so we are avoiding these singularities.
        
        Furthermore, $\Re(s+w)\geq\sigma+\alpha=\frac{1}{2k}+\varepsilon$. So we are also avoinding the poles of the function $P(s+w)$ which, as already discussed, are pure imaginary numbers.

        For $\varepsilon$ small enough, since $k\geq 2$, we have $\alpha<0$. As a consequence of that, we need to factor in the residue at $w=0$ and that is the only singularity we need to take into account when changing the contour.

        Before estimating the integral on the RHS of \eqref{first-perron-application-keven}, we first give a rough bound for $P(s)$. Since $\Re(s+w)>0$, we have
            \begin{align*}
                P(s+w)&=\prod_{p\mid q}\bigg(1-\frac{1}{p^{s+w}}\bigg)^{-1}\ll \prod_{p\mid q}\bigg(1-\frac{1}{p^{1/(2k)+\varepsilon}}\bigg)^{-1}\\
                &\ll \prod_{p\mid q}\bigg(1-\frac{1}{p^\varepsilon}\bigg)^{-1}\ll_\varepsilon \prod_{p\mid q}p^\varepsilon\ll_\varepsilon q^\varepsilon.
            \end{align*}

        Now since the residue at $w=0$ is simply $\frac{P(s)}{\zeta(ks)}$, by Cauchy's Integral Formula and the Residue Theorem, equation \eqref{first-perron-application-keven} can be estimated as
            \begin{align*}
                \frac{P(s)}{\zeta(ks)}-\sum_{n\leq y}\frac{h(n)}{n^s}\ll\bigg(\int_{1+iV}^{\alpha+iV}+\int_{\alpha+iV}^{\alpha-iV}+\int_{\alpha-iV}^{1-iV}\bigg)\frac{P(s+w)}{\zeta(k(s+w))}\frac{y^w}{w}\,dw+R(y).
            \end{align*}

        We use the bound for $1/\zeta$ that comes as a consequence of the proof that the Riemann Hypothesis implies the Lindel\"of Hypothesis (see \Cref{remarkLH}) to estimate both horizontal line segments (the first and third integrals). We have
            \begin{align*}
                \bigg|\int_{1+iV}^{\alpha+iV}\frac{P(s+w)}{\zeta(k(s+w))}\frac{y^w}{w}\,dw\bigg|&\ll_{q,\varepsilon} \frac{|t+V|^\varepsilon}{V}\int_\alpha^1 y^u\,du\ll_{q,\varepsilon} y|t|^\varepsilon V^{\varepsilon-1}.
            \end{align*}

        Over the vertical line segment, we have
            \begin{align*}
                \bigg|\int_{\alpha-iV}^{\alpha+iV}\frac{P(s+w)}{\zeta(k(s+w))}\frac{y^w}{w}\,dw\bigg|\ll_{q,\varepsilon}y^\alpha\int_{-V}^V\frac{|t+v|^\varepsilon}{|\alpha+iv|}\,dv\ll_{q,\varepsilon}y^{1/(2k)-\sigma+\varepsilon}|t|^\varepsilon V^\varepsilon.
            \end{align*}

        Choosing $V=y^2$ and putting the above estimates together, we obtain
            \begin{align*}
                \frac{P(s)}{\zeta(ks)}-\sum_{n\leq y}\frac{h(n)}{n^s}\ll_{q,\varepsilon}y^{1/(2k)-\sigma+3\varepsilon}|t|^\varepsilon+\frac{(y^2|t|)^\varepsilon}{y}+ \frac{\log y}{y} \ll_{q,\varepsilon} y^{1/(2k)-\sigma+3\varepsilon}|t|^\varepsilon,
            \end{align*}
        for any $\varepsilon>0$, as $y\to\infty$.            
    \end{proof}

\subsection{Proposition when \texorpdfstring{$k$}{k} is odd}

The proposition of the case when $k$ is odd is very similar to the preceding, so we pass through it with less details.

Let $\sum_{n=1}^\infty\frac{\Tilde{h}(n)}{n^s}$ be the Dirichlet series of $\frac{P(s)}{L(ks,\chi)P(ks)}$. For a fixed $y\geq 1$ and in the half-plane $\Re(s)\geq\frac{1}{k}+\delta,\delta>0$, define
    \begin{align*}
        \Tilde{H}_y(s)=\frac{P(s)}{L(ks,\chi)P(ks)}-\sum_{n\leq y}\frac{\Tilde{h}(n)}{n^s}.
    \end{align*}
We again are going to estimate the tail above.

Before we continue to the last proposition's proof, we recall a few notations already presented throughout the paper and fix some further. We defined the set $\mathcal{N}=\{n : p\mid n \implies p\mid q\}$ and the following function:
    \begin{align*}
        \nu(n)=
            \begin{cases}
                \mu(d)  &\text{  if $n=d^k$,}\\
                0       &\text{  otherwise.}\\
            \end{cases}
    \end{align*}

Therefore, $P(s)$ is the generating series of $\mathbbm{1}_{\mathcal{N}}(n)$ and $\frac{1}{P(ks)}$ is the generating series of the function $\nu(n)\mathbbm{1}_{\mathcal{N}}(n)$.

For simplicity, if $\Re(s)>\frac{1}{k}$, we call $\psi(n)$ the coefficients of $\frac{1}{L(ks,\chi)}$ and they are given by the following rule:
    \begin{align*}
        \psi(n)=
            \begin{cases}
                \mu(d)\chi(d)   &\text{ when $n=d^k$,}\\
                0               &\text{ otherwise.}
            \end{cases}
    \end{align*}
That is, similar to $\nu(n)$, the function $\psi(n)$ has support on the perfect $k$-power numbers.

Hence, the coefficients $\Tilde{h}(n)$ of the series $\frac{P(s)}{L(ks,\chi)P(ks)}$ are given by the double Dirichlet convolution $( \psi\ast\nu\mathbbm{1}_{\mathcal{N}}\ast\mathbbm{1}_{\mathcal{N}})(n)$.

    \begin{proposition}\label{proposition-kodd}
        Let $\Tilde{H}_s$ be defined as above and assume the Generalized Riemann Hypothesis. For $\Re(s)\geq\frac{1}{k}+\varepsilon$,
            \begin{align*}
                \Tilde{H}_y(s)\ll_{q,\varepsilon}y^{1/(2k)-\sigma+\varepsilon}|t|^\varepsilon,\forall\varepsilon>0,
            \end{align*}
        as $y\to\infty$.
    \end{proposition}
    \begin{proof}
        We apply Perron's formula with the same choice of parameters as the even case. That is, let $\sigma\geq \frac{1}{k}+\varepsilon$, $w=u+iv$ and we integrate over the line segment with real part $\sigma_0=1$. By \Cref{perron-formula-montgomery-vaughan}, we have
            \begin{align}\label{perron-aplication-kodd}
                \sum_{n\leq y}\frac{\Tilde{h}(n)}{n^s}=\frac{1}{2\pi i}\int_{1-iV}^{1+iV}\frac{P(s+w)}{L(k(s+w),\chi)P(k(s+w))}\frac{y^w}{w}\,dw+R(y).
            \end{align}
        
        Assuming the Generalized Riemann Hypothesis, we have that $\frac{1}{L(k(s+w),\chi)}$ has no poles if $\Re(k(s+w))\geq \frac{1}{2}+\varepsilon$, for any $\varepsilon>0$. So we again let $\alpha=\frac{1}{2k}-\sigma+\varepsilon$ and shift the contour of integration to the piecewise linear path connected by the points: $1+iV$, $\alpha+iV$, $\alpha-iV$ and $1-iV$. It follows that, over the new path of integration, $\Re(k(s+w))\geq\frac{1}{2}+k\varepsilon$.

        Furthermore, as we already discussed, the Generalized Riemann Hypothesis implies the Lindel\"of Hypothesis for Dirichlet $L$-functions (see \cite{Iwaniec-kowalski}, pp. 101 and 116) and we use this when estimating our upper bound for $\eqref{perron-aplication-kodd}$. 

        Now, besides the bound for $P(s+w)$ that we computed in the even case, we note that we have the following:
            \begin{align*}
                \frac{1}{P(k(s+w))}&=\prod_{p\mid q}\bigg(1-\frac{1}{p^{k(s+w)}}\bigg)\ll \prod_{p\mid q}\bigg(1+\frac{1}{p^{\frac{1}{2}+k\varepsilon}}\bigg)\ll \bigg(1+\frac{1}{2^{\frac{1}{2}+k\varepsilon}}\bigg)^{\omega(q)}\ll_q 1.
            \end{align*}

        The estimate for both horizontal line segments is
            \begin{align*}
                \bigg|\int_{\alpha+iV}^{1+iV}\frac{P(s+w)}{L(k(s+w),\chi)P(k(s+w))}\frac{y^w}{w}\,dw\bigg|\ll_{q,\varepsilon}y|t|^\varepsilon V^{\varepsilon-1}.
            \end{align*}

        For the vertical line segment, we have
            \begin{align*}
                \bigg|\int_{\alpha-iV}^{\alpha+iV}\frac{P(s+w)}{L(k(s+w),\chi)P(k(s+w))}\frac{y^w}{w}\,dw\bigg|\ll_{q,\varepsilon}y^{1/(2k)-\sigma+\varepsilon}|t|^\varepsilon V^\varepsilon.
            \end{align*}

        The error term contributes $R(y)\ll\frac{y\log y}{V}$ and, for small enough $\varepsilon$, the residue at $w=0$ is $\frac{P(s)}{L(ks,\chi)P(ks)}$. Choosing $V=y^2$, the above estimates sums up to
            \begin{align*}
                \frac{P(s)}{L(ks,\chi)P(ks)}-\sum_{n\leq y}\frac{\Tilde{h}(n)}{n^s}\ll_{q,\varepsilon} y^{1/(2k)-\sigma+3\varepsilon}|t|^\varepsilon,\forall\varepsilon>0,\forall\Re(s)\geq\frac{1}{k},
            \end{align*}
        as $y\to\infty$.
    \end{proof}

\section{Proof of \texorpdfstring{\Cref{main-theorem}}{}}

\subsection{Proof when \texorpdfstring{$k$}{k} is even}

We begin the proof by splitting the partial sums of the function $f(n)$ and estimating separately.

A small remark: the Dirichlet series of $L(s,\chi)H_y(s)$, defined at first for $\Re(s)>1$, has partial sums given by $\sum_{\substack{ab\leq x \\ a>y}}h(a)\chi(b)$. To see this, letting $\mathcal{A}=\{n : n>y\}$, we have
    \begin{align*}
           L(s,\chi)H_y(s)=L(s,\chi)\sum_{n>y}\frac{h(n)}{n^s}=L(s,\chi)\sum_{n=1}^\infty\frac{h(n)\mathbbm{1}_{\mathcal{A}}(n)}{n^s}=\sum_{n=1}^\infty\frac{(h\mathbbm{1}_{\mathcal{A}}\ast \chi)(n)}{n^s},
    \end{align*}
by the rules of the Dirichlet convolution.

On the other side,
    \begin{align*}
        \sum_{n\leq x}(h\mathbbm{1}_{\mathcal{A}}\ast \chi)(n)=\sum_{n\leq x}\sum_{ab=n}h(a)\mathbbm{1}_{\mathcal{A}}(a)\chi(b)=\sum_{\substack{ab\leq x \\ a>y}}h(a)\chi(b).
    \end{align*}

    \begin{proof}[Proof of \Cref{main-theorem}]
        Let $y<x$ be fixed. Since $\sum_{n=1}^{\infty} \frac{f(n)}{n^s}=L(s,\chi)\frac{P(s)}{\zeta(ks)}$, for $\Re(s)>1$, the coefficients $f(n)$ of this series are given by the Dirichlet convolution $f(n)=(\chi\ast h)(n)=\sum_{ab=n}h(a)\chi(b)$. Then we have
            \begin{align*}
                \sum_{n\leq x}f(n)&\coloneqq \sum_{n\leq x}\mu^{(k)}(n)g_\chi(n)=\sum_{n\leq x}\sum_{ab=n}h(a)\chi(b)=\sum_{ab\leq x}h(a)\chi(b)\\
                &=\sum_{\substack{ab\leq x \\ a\leq y}}h(a)\chi(b)+\sum_{\substack{ab\leq x \\ a>y}}h(a)\chi(b)\\
                &\coloneqq A+B.
            \end{align*}

        \noindent\textit{Estimate for $A$.}
        
        Since $\chi$ has modulus $q$, by periodicity of the character, we have $|\sum_{n\leq x}\chi(n)|\leq q$.

        Now, as \Cref{lemma-abs-value-inside-sum} gives the upper bound $\sum_{n\leq x}|h(n)|\ll x^{1/k}(\log x)^{\pi(q)}$, it follows that
            \begin{align*}
                |A|&\coloneqq\bigg|\sum_{\substack{ab\leq x \\ a\leq y}}h(a)\chi(b)\bigg|\leq\sum_{a\leq y}|h(a)|\bigg|\sum_{b\leq \frac{x}{a}}\chi(b)\bigg|\\
                &\ll_q \sum_{a\leq y}|h(a)|\ll_q y^{1/k}(\log x)^{\pi(q)}\ll_{q,\varepsilon} y^{1/k}x^\varepsilon.
            \end{align*}

        \noindent\textit{Estimate for $B$.}
        
        To estimate $B$ we use \Cref{perron-formula-montgomery-vaughan} again, with $\sigma_0=1+\frac{1}{\log x}$. Therefore, we have
            \begin{align*}
                |B|\coloneqq \bigg|\frac{1}{2\pi i}\int_{\sigma_0-iT}^{\sigma_0+iT}L(s,\chi)H_y(s)\frac{x^s}{s}\, ds+R(x)\bigg|,
            \end{align*}
        where
            \begin{align*}
                R(x)\ll \sum_{\substack{x/2<n<2x \\ n\neq x}}|f(n)|\min\bigg\{1,\frac{x}{T|x-n|}\bigg\}+\frac{4^{\sigma_0}+x^{\sigma_0}}{T}\sum_{n=1}^\infty \frac{|f(n)|}{n^{\sigma_0}}.
            \end{align*}
        Noting that $|f(n)|=|\mu^{(k)}(n)g_\chi(n)|\leq 1,\forall n$, and since $\sigma_0>1$, the second sum on the RHS above is $\ll\sum_{n=1}^\infty\frac{1}{n^{\sigma_0}}$, thus it is $\ll\frac{1}{\sigma_0-1}=\log x$. For the first sum, we follow the same strategy as in the proof of \Cref{main-prop-keven}: when $n$ is close to $x$, i.e. $|n-x|<1$, we choose the first member of the minimum. For the other values of $n$, we choose $\frac{x}{T|x-n|}$ as the minimum. Therefore the contribution of $R(x)$ is:
            \begin{align*}
                R(x)\ll 2+\frac{x}{T}\sum_{1\leq m\leq x}\frac{1}{m}+\frac{x\log x}{T}\ll \frac{x\log x}{T}.
            \end{align*}

        Let $\beta\geq\frac{1}{2}+\varepsilon$ (for $k\geq 3$ we could take precisely $\beta\geq \frac{1}{2}$). We again shift the contour to a piecewise linear path that goes from $\sigma_0+iT$ to $\beta+iT$, the line segment that goes from $\beta+iT$ to $\beta-iT$ and from $\beta-iT$ to $\sigma_0-iT$. Note that we are not passing any singularities of the integrand. Hence, we can estimate each linear path as follows: for the horizontal line segments, since $y<x$, we have
            \begin{align*}
                \bigg|\int_{\sigma_0+iT}^{\beta+iT}L(s,\chi)H_y(s)\frac{x^s}{s}\, ds\bigg|&\ll_{q,\varepsilon} T^{1/4-1}y^{1/(2k)+\varepsilon}T^\varepsilon\int_{\sigma_0}^\beta \bigg(\frac{x}{y}\bigg)^\sigma \,d\sigma\\
                &\ll_{q,\varepsilon}T^{\varepsilon-3/4}y^{1/(2k)+\varepsilon}\bigg(\frac{x}{y}\bigg)^{\sigma_0}\\
                &\ll_{q,\varepsilon}xy^{1/(2k)-\sigma_0+\varepsilon}T^{\varepsilon-3/4},
            \end{align*}
        where we used\footnote{Note that for this part of the theorem, we don't really need the Lindel\"of Hypothesis for Dirichlet $L$-functions. It is enough to use the weaker convexity bound, i.e. $L(s,\chi)\ll (1+|t|)^{1/4}$ (see \cite{Iwaniec-kowalski}, Theorem 5.23, p. 119). However, since for the odd case the Generalized Riemann Hypothesis is essential for our proof, we will apply the Lindel\"of Hypothesis later for the next part of the proof, when possible.} the convexity bound for $L(s,\chi)$ and \Cref{main-prop-keven}.

        For the vertical estimate, using \Cref{lema-integral-L} and \Cref{main-prop-keven}, we have
            \begin{align*}
                \bigg|\int_{\beta-iT}^{\beta+iT}L(s,\chi)H_y(s)\frac{x^s}{s}\, ds\bigg|&\ll_{q,\varepsilon} y^{1/(2k)-\beta+\varepsilon}x^{\beta}\int_{-T}^{T}|t|^\varepsilon \frac{|L(\beta+it,\chi)|}{|\beta+it|}\, dt\\
                &\ll_{q,\varepsilon} y^{1/(2k)-\beta+\varepsilon}x^{\beta}T^\varepsilon(\log T)^{3/2}\\
                &\ll_{q,\varepsilon}x^{\beta}y^{1/(2k)-\beta+\varepsilon}T^{2\varepsilon}.
            \end{align*}

        Combining these estimates and choosing $T=x^2$,
            \begin{align*}
                |B|\ll_{q,\varepsilon} x^{\beta+4\varepsilon}y^{1/(2k)-\beta+\varepsilon}.
            \end{align*}

        \noindent\textit{Estimate for $A+B$.}

        Recalling that we require $y<x$, but otherwise we have freedom to choose $y$ as a function of $x$, we take $y=x^{\frac{2k\beta}{2k\beta+1}}$. Hence,
            \begin{align*}
                \sum_{n\leq x}f(n)\ll_{q,\varepsilon} x^{\frac{2\beta}{2k\beta+1}+\varepsilon}+x^{(\frac{2k\beta}{2k\beta+1})(\frac{1}{2k}-\beta+\varepsilon)+\beta+4\varepsilon}\ll_{q,\varepsilon}x^{\frac{2\beta}{2k\beta+1}+4\varepsilon},
            \end{align*}
        for any $\varepsilon>0$, as $x\to\infty$.
        
        Finally we note that this bound is optimal when $\beta$ is closer to $\frac{1}{2}$. Therefore,
            \begin{align*}
                \sum_{n\leq x}f(n)\ll_{q,\varepsilon}x^{1/(k+1)+4\varepsilon},\forall\varepsilon>0,
            \end{align*}
        and we conclude the proof for the even case.
    \end{proof}

\subsection{Proof when \texorpdfstring{$k$}{k} is odd}

This case is not as straightforward as the even case, since we need a bit more care when dealing with the sum in which $a\leq y$. Otherwise it follows the same steps.

    \begin{proof}[Proof of \Cref{main-theorem}]
    
        Let $y<x$ be fixed. Employing the same method as the previous case, we have
            \begin{align*}
                \sum_{n\leq x}f(n)=\bigg(\sum_{\substack{ab\leq x \\ a\leq y}}+\sum_{\substack{ab\leq x \\ a>y}}\bigg)\Tilde{h}(a)\chi(b)\coloneqq A+B.
            \end{align*}

        \noindent\textit{Estimate for $A$.}
        
        Using the notations we fixed before \Cref{proposition-kodd}, we start with the following:
            \begin{align*}
                |(\nu\mathbbm{1}_\mathcal{N}\ast \mathbbm{1}_\mathcal{N})(n)|=\bigg|\sum_{\substack{a^kb=n \\ a,b\in\mathcal{N}}}\mu(a)\bigg|\leq \sum_{\substack{a^kb=n \\ a,b\in\mathcal{N}}}|\mu(a)|\leq \sum_{\substack{a^kb=n \\ b\in\mathcal{N}}}|\mu(a)|,
            \end{align*}
        where in the last inequality we just relaxed the condition on $a$. Consequently, we have
            \begin{align*}
                \sum_{n\leq x}|(\nu\mathbbm{1}_\mathcal{N}\ast \mathbbm{1}_\mathcal{N})(n)|\leq \sum_{n\leq x}\sum_{\substack{a^kb=n \\ b\in\mathcal{N}}}|\mu(a)|\leq \sum_{a\leq x^{1/k}}\sum_{\substack{b\leq \frac{x}{a^k} \\ b\in \mathcal{N}}}1\ll_q x^{1/k}(\log x)^{\pi(q)}.
            \end{align*}
        Note that the last double sum doesn't differ from the right hand side of \eqref{equation-modulus-h} and it doesn't rely on the parity of $k$. Thus, we obtain the same bound as before.

        We remember that we want to estimate the coefficients of the product of three Dirichlet series, that is, $\frac{P(s)}{L(ks,\chi)P(ks)}$. Using the definition of Dirichlet convolution again, we have
            \begin{align*}
                \sum_{n\leq U}|\Tilde{h}(n)|&\coloneqq\sum_{n\leq U}|(\psi\ast\nu\mathbbm{1}_\mathcal{N}\ast \mathbbm{1}_\mathcal{N})(n)|=\sum_{n\leq U}\bigg|\sum_{md=n}\psi(d)(\nu\mathbbm{1}_\mathcal{N}\ast \mathbbm{1}_\mathcal{N})(m)\bigg|\\
                &\leq \sum_{n\leq U}\sum_{md=n}|\psi(d)||(\nu\mathbbm{1}_\mathcal{N}\ast \mathbbm{1}_\mathcal{N})(m)|=\sum_{d\leq U}|\psi(d)|\sum_{m\leq\frac{U}{d}}|(\nu\mathbbm{1}_\mathcal{N}\ast \mathbbm{1}_\mathcal{N})(m)|\\
                &\ll_q \sum_{d\leq U^{1/k}}|\mu(d)\chi(d)|\bigg(\frac{U}{d^k}\bigg)^{1/k}(\log U)^{\pi(q)}\ll_q U^{1/k}(\log U)^{\pi(q)}\sum_{d\leq U^{1/k}}\frac{1}{d}\\
                &\ll_q U^{1/k}(\log U)^{\pi(q)+1}.
            \end{align*}

        Therefore, for $A$, we have
            \begin{align*}
                |A|\leq \sum_{a\leq y}\bigg|\Tilde{h}(a)\sum_{b\leq \frac{x}{a}}\chi(b)\bigg|\ll_q \sum_{a\leq y}|\Tilde{h}(a)|\ll_{q,\varepsilon} y^{1/k}x^\varepsilon.
            \end{align*}

        \noindent\textit{Estimate for $B$.}
        
        We use Perron's formula again, integrating with $\sigma_0=1+\frac{1}{\log x}$. We have,
            \begin{align*}
                |B|\coloneqq\bigg|\frac{1}{2\pi i}\int_{\sigma_0-iT}^{\sigma_0+iT}L(s,\chi)\Tilde{H}_y(s)\frac{x^s}{s}\,ds+R(x)\bigg|,
            \end{align*}
        where $R(x)\ll\frac{x\log x}{T}$, by the exactly same argument as the even case.

        And just as before, we shift the contour to the path going from $\sigma_0+iT$ to $\beta+iT$, from $\beta+iT$ to $\beta-iT$ and $\beta-iT$ to $\sigma_0-iT$, where $\beta\geq \frac{1}{2}+\varepsilon$ is fixed and estimate the contribution of each line segment together with the error term. We obtain that
            \begin{align*}
                B\ll_{q,\varepsilon} x^{\beta+\varepsilon} y^{1/(2k)-\beta+\varepsilon},
            \end{align*}
        for any $\varepsilon>0$. Here we are omitting the details but we stress that throughout the proof we use the fact that the Generalized Riemann Hypothesis implies the Lindel\"of Hypothesis for $L$-functions, the application of \Cref{proposition-kodd} and we choose the parameter $T$ appropriately.

        \noindent\textit{Estimate for $A+B$.}
        
        Putting the estimates together and taking $y=x^{\frac{2k\beta}{2k\beta+1}}$, we have
            \begin{align*}
                \sum_{n\leq x}f(n)\ll_{q,\varepsilon}x^{\frac{2\beta}{2k\beta+1}+\varepsilon},
            \end{align*}
        for any $\varepsilon>0$.

        Since the above is optimized when $\beta=\frac{1}{2}+\varepsilon$, the result follows.
    \end{proof}
    
\section{Acknowledgments}

The author would like to thank his advisor, Marco Aymone, for his guidance and many helpful suggestions and discussions, and also to the anonymous referee for the suggestions and comments. This project is funded in part by the Coordenação de Aperfeiçoamento de Pessoal de Nível Superior - Brasil (CAPES), FAPEMIG grant Universal no. APQ-00256-23 and by CNPq grant Universal no. 403037 / 2021-2.

\bibliographystyle{siam}
\bibliography{main.bib}

\end{document}